\documentclass[a4paper,reqno]{amsart}

\usepackage[margin=3.7cm]{geometry}

\usepackage{mathtools}
\usepackage{amssymb}
\usepackage{amsthm}
\usepackage{thmtools} 
\usepackage{slashed}

\usepackage[
	bookmarks=true,
    pdfusetitle,
    hidelinks
    ]{hyperref}

\numberwithin{equation}{section}

\newtheorem{thm*}{Theorem}
\newtheorem{thm}{Theorem}[section]

\newtheorem{cor}[thm]{Corollary}

\newtheorem{lemma}[thm]{Lemma}
\newtheorem{prop}[thm]{Proposition}

\theoremstyle{definition}
\newtheorem{definition}[thm]{Definition}

\theoremstyle{remark}
\newtheorem{remark}[thm]{Remark}

\newcommand{\rd}{\mathrm{d}}

\newcommand{\C}{\mathbb{C}}

\newcommand{\R}{\mathbb{R}}
\newcommand{\N}{\mathbb{N}}

\DeclarePairedDelimiterX{\set}[1]{\lbrace}{\rbrace}{\,#1\,}
\DeclarePairedDelimiterX{\setcond}[2]{\lbrace}{\rbrace}{\,#1\,\mathclose{} :\mathopen{}\,#2\,}

\DeclareMathOperator{\divergence}{div}

\DeclareMathOperator{\Dom}{Dom}
\DeclareMathOperator{\Ran}{ran}
\DeclareMathOperator{\Span}{Span}
\DeclareMathOperator{\sign}{sign}

\DeclareMathOperator{\real}{Re} 

\DeclarePairedDelimiter{\braket}{\langle}{\rangle}
\DeclarePairedDelimiter{\norm}{\lVert}{\rVert}

\date{\today}

\title{Friedlander's inequality and the de Rham complex}

\author[M. Fries]{Magnus Fries}
\author[M. Goffeng]{Magnus Goffeng}
\author[G. Miranda]{Germ\'an Miranda}

\address[M. Fries]{Department of Mathematics, Lund University, Sweden}
\email{magnus.fries@math.lth.se}
\address[M. Goffeng]{Department of Mathematics, Lund University, Sweden}
\email{magnus.goffeng@math.lth.se}
\address[G. Miranda]{Department of Mathematics, Lund University, Sweden}
\email{german.miranda@math.lth.se}

\subjclass[2020]{
	\href{https://mathscinet.ams.org/mathscinet/msc/msc2020.html?t=35P15}{35P15}, 
	\href{https://mathscinet.ams.org/mathscinet/msc/msc2020.html?t=58J10}{58J10}, 
	\href{https://mathscinet.ams.org/mathscinet/msc/msc2020.html?t=58J50}{58J50}
	}

\begin{document}

\begin{abstract}
Inequalities between Dirichlet and Neumann eigenvalues of the Laplacian and of other differential operators have been intensively studied in the past decades. The aim of this paper is to introduce differential forms and the de Rham complex in the study of such inequalities. We show how differential forms lie hidden at the heart of the work of Rohleder on inequalities between Dirichlet and Neumann eigenvalues for the Laplacian on planar domains. 
\end{abstract}

\maketitle 

\section{Introduction}
Let \(\Omega \subset \mathbb{R}^d\) be an open, connected, bounded domain with Lipschitz boundary. We write \(\lambda_j(T)\) for the \(j\):th eigenvalue, ordered increasingly, for a positive operator \(T\) with discrete spectrum. We denote by \(\Delta_D\) and \(\Delta_N\) the Dirichlet and Neumann realization of the Laplacian on \(\Omega\). The main goal of this paper is to introduce methods of differential forms and the de Rham complex as a tool for obtaining inequalities between eigenvalues of \(\Delta_D\) and \(\Delta_N\). 

The inequality \(\lambda_2(\Delta_N) <\lambda_1(\Delta_D)\) appears already in the work of P\'olya for \(d=2\) \cite{Polya}, and it was extended by Payne to \(\lambda_{j+2}(\Delta_N) <\lambda_{j}(\Delta_D)\) for all \(j=1,2, \ldots\) when \(\Omega\) is a \(C^2\) convex domain \cite{Payne}. Later, Levine and Weinberger generalized the inequality \(\lambda_{j+d}(\Delta_N)<\lambda_{j}(\Delta_D)\) considering a convex domain in \(\mathbb{R}^d\) with \(C^2\) boundary with H\"older continuous second derivatives \cite{LevineWeinberger}. As pointed out by Levine and Weinberger, the previous inequality can be extended by approximation to \(\lambda_{j+d}(\Delta_N)\leq \lambda_{j}(\Delta_D)\) for all convex bounded domains. Moreover, they recovered the inequality \(\lambda_{j+1}(\Delta_N) < \lambda_j (\Delta_D)\) for domains with \(C^2\) boundary and non-negative mean curvature proven by Aviles \cite{Aviles}. 

In 1991, Friedlander \cite{Friedlander} used properties of the Dirichlet-to-Neumann operator to prove 
\begin{equation}
	\label{eq:fried} 
	\lambda_{j+1}(\Delta_N) \leq \lambda_j(\Delta_D),
\end{equation}
for all \(j=1,2,\ldots,\) and for \(C^1\) domains and no curvature assumption. The smoothness assumption was removed and the inequality was proven to be strict by Filonov in \cite{Filonov} in a beautiful argument using Glazman's lemma, see more in the textbook \cite{SpectralGeometryBook}. We call Equation \eqref{eq:fried} Friedlander's inequality. Recently, Rohleder \cite{rohlfr} proved that for any simply connected, bounded, Lipschitz domains in \(\R^2\), there is an inequality 
\begin{equation}
	\label{eq:rohlineq}
	\lambda_{j+2}(\Delta_N)\leq \lambda_{j}(\Delta_D),
\end{equation}
for any \(j\in \N\). We combine the ideas of \cite{rohlfr} with the de Rham complex \cite{leschbruening} into a common framework that in arbitrary dimension allows a generalization of Rohleder's results on the curl curl operator \cite{RohlederCurlCurl}, as well as a short proof of Rohleder's inequality \eqref{eq:rohlineq} in dimension two. 

The main novelty in this paper is found in the method we introduce. As mentioned, the method stems in the de Rham complex. Since our geometries have boundaries, we require an appropriate boundary condition. We use the so called absolute boundary condition allowing us to rely on previous work \cite{leschbruening} on this well studied Hilbert complex. What is promising with this method is that it allows us to give concise proofs of work of Rohleder \cite{rohlfr,RohlederCurlCurl} and carries large dimensions of internal degrees of freedom that holds hope of pushing the estimates even further. 

It was conjectured that 
\begin{equation}
\label{eq:conjecture}
	\lambda_{j+d}(\Delta_N) < \lambda_j (\Delta_D)
\end{equation}
holds for all domains \(\Omega\subset \mathbb{R}^d\) with no convexity assumption \cite{BenguriaLevitinParnovski}. If \(d=2,3\), \eqref{eq:conjecture} is sharp because the unit ball is an edge case, namely \(\lambda_{d+2}(\Delta_N) > \lambda_1(\Delta_D)\). For \(d\geq 4\), it was observed in \cite{cox2019isoperimetricrelationsdirichletneumann} that for the unit ball we have more than \(d+1\) Neumann eigenvalues strictly below the first Dirichlet eigenvalue. For example, if \(d=4\) we have \(\binom{d}{d-1} + \binom{d+1}{d-1}\) Neumann eigenvalues strictly below the first Dirichlet eigenvalue. Recall that, for the unit ball, these binomial coefficients are connected with the dimension of the space of spherical harmonics of certain degrees (see \cite[Section 1.2.3]{SpectralGeometryBook} for more details). In a recent work, Freitas\cite{freitas2024growinggapdirichletneumann} studied the gap between Dirichlet and Neumann eigenvalues with respect to the index \(j\). It was conjectured that
	\begin{equation}\label{eq:freitasShift}
		\lambda_{j + \lfloor c(d,j)\rfloor}(\Delta_N) \leq \lambda_j (\Delta_D),
	\end{equation}
	where \(c(d,j) = \frac{d V_{d-1}}{2V_{d}^{1-2/d}} j^{1-1/d}\) and \(V_d\) is the volume of the d-dimensional unit ball. Moreover, for \(d\geq 4\) and all \(j\in \N\) we have the weaker inequality
	\begin{equation}
		\lambda_{j+\lfloor C_{\Omega} j^{1-3/d}\rfloor} (\Delta_N) \leq \lambda_j(\Delta_D),
	\end{equation}
	where the constant \(C_{\Omega}\) is not explicit (see \cite{SafarovFilonov,freitas2024growinggapdirichletneumann}). The de Rham complex introduces large binomial coefficients into the estimates that we hope can provide insight into the conjectural extra shift \eqref{eq:conjecture} for general domains.

The techniques used in proving inequalities for the Laplacian have been adapted to other differential operators. Frank, Helffer and Laptev \cite{LaptevFrank, frank2024inequalitiesdirichletneumanneigenvalues} adapted such ideas to prove a similar inequality for the sub-Laplacian on an open set of a Carnot group, which in particular covers the Heisenberg group. Another example given by Mazzeo \cite{Mazzeo} is the adaptation of Friedlander's ideas to prove the same inequality for certain manifolds, e.g. for all symmetric spaces of noncompact type. However, for manifolds there are cases where the inequality \eqref{eq:fried} does not hold, for example any spherical cap larger than a hemisphere \cite{Mazzeo}. Recently, Lotoreichik explored these inequalities for the magnetic Laplacian with the homogeneous magnetic field in two and three dimensions \cite{lotoreichik2024inequalitiesdirichletneumanneigenvalues}. Similar inequalities have also been proven for Schr\"odinger operators \(-\Delta+V\) under convexity assumptions and further restrictions on the potential \cite{RohlederSchrodinger}, and between the eigenvalues of a curl curl operator and the Dirichlet Laplacian \cite{RohlederCurlCurl}.

The paper is organized as follows. In \autoref{sec:Preliminaries} we give a brief summary of the main tools we need in order to introduce the de Rham complex and prove the main results. We recall the de Rham complex on a manifold with boundary as well as the relevant technical results thereon in \autoref{sec:derahm}. In \autoref{sec:SpectralPropdeRham} we rephrase the results of \cite{rohlfr} in general dimension using the de Rham complex. We compare our methods to Rohleder's work \cite{rohlfr,RohlederCurlCurl} in dimension 2 and 3 in \autoref{sec:Rohleder_language} where we provide a short proof for Rohleder's inequality \eqref{eq:rohlineq}.

There is a recent preprint \cite{mikhail2025hodgelaplacianeigenvaluessurfacesboundary} with results overlapping part of the results in this paper. In particular Theorem 1.5 in \cite{mikhail2025hodgelaplacianeigenvaluessurfacesboundary} can be obtained as a consequence of \autoref{lemma:rohledergeneralization} for smooth domains in a similar way as we did for dimensions 2 and 3 in \autoref{sec:Rohleder_language}. In both cases, differential forms are used to obtain the results, but the techniques are different. A previous version of this work, published as \cite{FGMFriedlanderreftoJST}, contained an erroneous claim that has been removed from this version.

\section{Preliminaries}
\label{sec:Preliminaries}

Let \(T\) be a positive, self-adjoint operator with discrete spectrum (i.e. the spectrum consists of isolated eigenvalues of finite multiplicity). We denote by \(\lambda_j(T)\) the \(j\):th eigenvalue of \(T\) ordered increasingly counting multiplicity. We write the counting function of \(T\) as
\[
	N(T,\lambda) := \# \setcond{j}{\lambda_j(T)\leq \lambda}. 
\]
We let
\[
	m(T, \lambda) := \dim \ker (T - \lambda),
\]
denote the multiplicity of an eigenvalue (or \(0\) if \(\lambda\) is not an eigenvalue). Note that
\begin{equation}
\label{multiident}
	j \leq N(T, \lambda_j (T)) \leq j + m(T, \lambda_j(T)) - 1.  
\end{equation}

\subsection{Variational principle}
The results in \cite{Filonov} and \cite{rohlfr} rely on a variational principle, which will also be useful in our approach using the de Rham complex. Because of this, we recall Glazman's lemma describing the counting function by means of finite-dimensional subspace of the form domain.

\begin{lemma}[Glazman's lemma]\label{lemma:Glazman}
	Let \(T\) be a positive self-adjoint operator with discrete spectrum acting on a Hilbert space \((\mathcal{H},\braket{\cdot, \cdot})\) and \(q_T\) the quadratic form associated with \(T\). Then
\[
	N(T, \lambda) = \max_{\substack{V\subseteq \Dom(q_T) \\ R[u]\leq \lambda \forall u\in V\setminus \{0\}}} \dim V,
\]
where \(R[u] = \frac{q_T(u)}{\braket{ u, u}}\) is the so-called Rayleigh quotient. 
\end{lemma}

See \cite[Proposition 9.5]{ShubinInvitation} for a proof of the Glazman's lemma. Hence, if we have a finite-dimensional subspace \(V\subseteq \Dom(q_T)\) such that 
\[ q_T(u)\leq \lambda \norm{u}^2, \]
for \(u\in V\), then using Glazman's lemma we obtain
\[ N(T,\lambda) \geq \dim V. \]
We will also use the notation \(N(q_T,\lambda) :=  N(T, \lambda)\). We use the notation \(T:\mathcal{H}\dashrightarrow \mathcal{H}'\) for a densely defined operator between two Hilbert spaces. A fact we use is that if there is a closed densely defined operator $t: \mathcal{H}\dashrightarrow \mathcal{H}'$ for
some Hilbert space $\mathcal{H}'$ such that $T = t^*t$, then $\Dom(q_T) = \Dom(t)$ and $q_T(u,v) = \langle tu,tv\rangle$.

\subsection{Hilbert Complexes}
\label{subsec:deRhamComplex}

A key aspect in our study of inequalities between Dirichlet and Neumann Laplacian eigenvalues will be the usage of the de Rham complex on a domain in \(\R^n\) with appropriate boundary conditions. It is helpful to set this in an abstract framework, so we first recall the notion of a Hilbert complex. We follow the presentation of \cite{leschbruening} and refer the reader there for further details. Below in \autoref{sec:derahm} we specialize to de Rham complexes. 

\begin{definition}
A Hilbert complex written as \((\mathcal{H}_\bullet,T_\bullet)\) or 
\[ 0\to \mathcal{H}_0\xrightarrow{T_0}\mathcal{H}_1\xrightarrow{T_1}\cdots \mathcal{H}_{d-1}\xrightarrow{T_{d-1}}\mathcal{H}_d\to 0, \]
consists of Hilbert spaces \(\mathcal{H}_0,\mathcal{H}_1,\ldots, \mathcal{H}_d\) and closed, densely defined maps \(T_k:\mathcal{H}_{k}\dashrightarrow \mathcal{H}_{k+1}\) with the property that 
\[ \Ran(T_{k-1})\subseteq \ker(T_k). \]
In other words, \(\Ran(T_{k-1})\subseteq \Dom(T_k)\) and \(T_kT_{k-1}=0\).

We say that \((\mathcal{H}_\bullet,T_\bullet)\) is Fredholm if the cohomology groups 
\[ H^k(\mathcal{H}_\bullet,T_\bullet) := \ker(T_k)/\Ran(T_{k-1}), \]
are finite-dimensional. The Euler characteristic of a Fredholm Hilbert complex \((\mathcal{H}_\bullet,T_\bullet)\) is defined as 
\begin{equation}
\label{eulerdefe}
	\chi(\mathcal{H}_\bullet,T_\bullet) := \sum_{k=0}^d (-1)^k \dim H^k(\mathcal{H}_\bullet,T_\bullet).
\end{equation}
We say that \((\mathcal{H}_\bullet,T_\bullet)\) has discrete spectrum if for any \(k\) the densely defined, self-adjoint Laplacians
\[
	\Delta_{k,T_\bullet} := T_k^*T_k+T_{k-1}T_{k-1}^*:\mathcal{H}_k\dashrightarrow\mathcal{H}_k,
\]
have discrete spectrum.
\end{definition}

We note it is a stronger assumption to have discrete spectrum than being Fredholm. Moreover, each operator \(T_k\) has closed range if the Hilbert complex \((\mathcal{H}_\bullet,T_\bullet)\) is Fredholm. 
Note also that \(\Delta_{k,T_\bullet} = (T_k + T_{k-1}^*)^*(T_k + T_{k-1}^*)\) and hence \(\Dom q_{\Delta_{k,T_\bullet}} = \Dom T_k \cap \Dom T_{k-1}^*\).
\\

We will utilize Hilbert complexes \((\mathcal{H}_\bullet, T_\bullet)\) in order to compare the spectrum of the bottom Laplacian \(\Delta_{0,T_\bullet}=T_0^*T_0\) and the top Laplacian \(\Delta_{d,T_\bullet} := T_dT_d^*\). We shall see below that for the de Rham complex with an appropriate boundary condition, \(\Delta_{0,T_\bullet}\) is the Neumann realization of the Laplacian and \(\Delta_{d,T_\bullet}\) is up to the Hodge star the Dirichlet realization of the Laplacian. Assuming that \((\mathcal{H}_\bullet, T_\bullet)\) is Fredholm, we have the Hodge decomposition
\begin{equation}
\label{hodgedecomp}
 \mathcal{H}_k = \ker(\Delta_{k,T_\bullet})\oplus \Ran(T_{k}^*) \oplus \Ran (T_{k-1}). 
 \end{equation}
In particular, if \((\mathcal{H}_\bullet, T_\bullet)\) in fact has discrete spectrum we can deduce that
\begin{equation}
\label{eq:hodgrealref}
N(\Delta_{k,T_\bullet}, \lambda) = \dim \ker(\Delta_{k,T_\bullet})+ N(T_k^*T_k;(0, \lambda])+ N(T_{k-1}T_{k-1}^*; (0,\lambda]).
\end{equation}
Here we use the notation \(N(T; (0,\lambda])\) for the number of eigenvalues of a self-adjoint operator \(T\) in the interval \((0,\lambda]\). Combining such terms in an alternating sum, and using that \(T_k^*T_k\) and \(T_kT_k^*\) has the same non-zero spectrum including multiplicities, we conclude the following lemma. 

\begin{lemma}
\label{thm:complex_alternating_sum}
Assume that \((\mathcal{H}_\bullet, T_\bullet)\) is a Hilbert complex with discrete spectrum. Then for any \(\lambda\geq 0\), we have an equality
\[ \sum_{k=0}^{d} (-1)^k N(\Delta_{k,T_\bullet}, \lambda)= \chi(\mathcal{H}_\bullet,T_\bullet). \]
\end{lemma}

\section{The de Rham complex}\label{sec:derahm}

We now turn to studying the de Rham complex with boundary conditions. The material in this section is well known and can be found in the literature \cite{leschbruening,gilkeyinv}. The reader uninitiated in differential forms and the de Rham complex can consult \cite{warnerbook} for more details, or the more elementary text \cite{guilleminhaine}. We take a smooth, oriented, compact manifold with Lipschitz boundary \(M\), or in other words, \(M\) is a precompact domain with Lipschitz boundary in a smooth, oriented manifold. We write \(d\) for the dimension of \(M\). To carry out spectral geometry, we need to choose a Riemannian metric \(g\) on \(M\). We denote the Riemannian volume form by \(\rd V\) and the Riemannian volume density by \(\rd x\). The volume density will only be used as a measure in integrals so its difference to differentials such as \(\rd x_j\) will be clear. We write \(T^*M\to M\) for the cotangent bundle on \(M\). Abusing the notation, we write \(\wedge^kT^*M\to M\) for the complexified bundle of degree \(k\)-forms on \(M\) and \(\wedge^*T^*M := \oplus_{k=0}^d\wedge^kT^*M\) for the bundle of all complexified forms on \(M\). 

For our application to the Friedlander's inequality \eqref{eq:fried}, we consider Lipschitz domains in \(\R^d\) with the Euclidean metric. {\bf For notational clarity, we reserve the letter \(\Omega\) for domains in \(\R^d\) and \(M\) for general manifolds.} If \(\Omega\) is a domain in \(\R^d\), the basis vectors of \(\R^d\) defines a frame and trivializations \(\wedge^kT^*\Omega\cong \Omega\times \wedge^k \C^d\).

The exterior differential between differential forms is a well studied differential operator. We write the exterior differential on forms of degree \(k\) as
\begin{equation}
\label{eq:exterior_differential}
	\rd_k:C^\infty(M,\wedge^k T^*M)\to C^\infty(M,\wedge^{k+1} T^*M)
\end{equation}
and the exterior differential on all forms as \(\rd:C^\infty(M,\wedge^* T^*M)\to C^\infty(M,\wedge^* T^*M)\). We also write \(\delta:C^\infty(M,\wedge^* T^*M)\to C^\infty(M,\wedge^* T^*M)\) for the formal adjoint of \(\rd\,\), decomposing over the form degrees into 
\begin{equation}\label{eq:delta_def}
\delta_k:C^\infty(M,\wedge^{k+1} T^*M)\to C^\infty(M,\wedge^{k} T^*M).
\end{equation}
The operator \(\delta_k\) takes the form 
\[ \delta_k=(-1)^{k+1}\star^{-1}\rd_{d-k-1} \star, \] 
where \(\star\) denotes the Hodge star. The Hodge star \(\star\omega\) of a \(k\)-form \(\omega\) is the \(n-k\)-form with the property that for any real \(k\)-form \(\omega'\),
\[ \omega'\wedge \star\omega=\braket{ \omega',\omega}_{\wedge^kT^*M}\rd V, \] 
where \(\braket{ \cdot,\cdot}_{\wedge^kT^*M}\) denotes the inner product on \(k\)-forms and \(\rd V\) denotes the Riemannian volume form. The operator \(\slashed{D} := \rd+\delta\) is an elliptic first order differential operator, called the Hodge-Dirac operator, and \(\slashed{D}^2\) coincides with the Hodge Laplacian on forms, see \cite{gilkeyinv}.

So far the discussion has only been concerned with differential expressions. Now we turn to realizations of these operators on \(L^2\)-spaces. We consider the Hilbert spaces
\[ \mathcal{H}_k := L^2(M;\wedge^k T^*M), \quad k=0,1,2,\ldots,d. \]
There are several ways to realize the exterior differential and its adjoint on \(\mathcal{H}_k\) as closed operators, notably via the \emph{ideal boundary conditions} defined as those boundary conditions ensuring that we obtain a Hilbert complex. Such boundary conditions are discussed in detail in \cite{leschbruening}. We will only use the so-called maximal realization but for completeness we also discuss the minimal realization. 
\begin{itemize}
\item We shall write \(\rd_{k,{\max}}\) for the maximal realization of \(\rd_k\), i.e. so \(u\in \Dom( \rd_{k,{\max}})\) if and only if \(u\in L^2(M;\wedge^k T^*M)\) satisfies that \(\rd u\in L^2(M;\wedge^{k+1}T^*M)\) where the exterior differential is applied in a distributional sense. The reader should beware that, for \(k>0\) the domain of \(\rd_{k,{\max}}\) is substantially larger than \(H^1(M;\wedge ^k T^*M)\). It follows from \cite[Section 4]{leschbruening} that \(H^1(M;\wedge ^k T^*M)\) is a core for \(\rd_{k,{\max}}\).
\item The minimal realization \(\rd_{k,{\min}}\) of \(\rd_k\), i.e. the graph closure of \(\rd_k\) acting on \(C^\infty_c(M^\circ, \wedge^k T^*M)\). Also in this case, the reader should be aware that for \(k<d\) the domain of \(\rd_{k,{\min}}\) is larger than \(H^1_0(M;\wedge^k T^*M)\) even if it follows from \cite[Section 4]{leschbruening} that \(H^1_0(M;\wedge^k T^*M)\) forms a core for the operator. 
\end{itemize}

Unless otherwise state, we use the maximal realization. We use the notation
\[ \Delta_{k,a} := \rd_{k,{\max}}^*\rd_{k,{\max}}+\rd_{k-1,{\max}}\rd_{k-1,{\max}}^*. \]
The index $a$ refers to its defining boundary condition which is called the \emph{absolute boundary condition}, it is called so for reasons that will become apparent in \autoref{absolueteapp} and \autoref{lkjnkjnkjanda} below. Note that \(\rd_{k-1,{\max}}^*\) is the minimal realization of \(\delta_{k-1}\). We make the following observations from quadratic form considerations. We have that 
\[ \Delta_{0,a}=\rd_{0,{\max}}^*\rd_{0,{\max}}=\Delta_N, \]
is defined from the quadratic form with domain $H^1(M)$ so it is the Neumann realization of the Hodge Laplacian on 0-forms. We have that 
\[ \Delta_{d,a}=\rd_{d,{\max}}\rd_{d,{\max}}^* = \Delta_D, \]
is defined from the quadratic form with domain $H^1_0(M,\wedge^dT^*M)$ so it is the Dirichlet realization of the Hodge Laplacian on \(d\)-forms. Indeed, the Hodge star \(L^2(M)\to L^2(M;\wedge^d T^*M)\), \(f\mapsto f\rd V\) implements a canonical identification of \(\Delta_{d,a}\) with the Dirichlet realization of the Laplacian on \(0\)-forms. Below in \autoref{thm:q_ka} we will see that the domain of \(\Delta_{k,a}\) is contained in \(H^1(M;\wedge^k T^*M)\), so by the Rellich theorem we obtain that the Hilbert complex \((L^2(M;\wedge^\bullet T^*M),\rd_\bullet)\) has discrete spectrum as soon as \(M\) is a compact manifold with Lipschitz boundary.

The operator \(\Delta_{k,a}\) is a realization of the Hodge Laplacian on \(k\)-forms, and the realization is described by a boundary condition. Let us clarify the boundary condition defining \(\Delta_{k,a}\) for \(0<k<d\) and describe their form domains. We do so using the Hodge--Dirac operator \(\slashed{D}=\rd+\delta\) and the results from \cite{leschbruening}. The results in  \cite{leschbruening} are described for smooth manifolds with boundary, but using \cite{MR0799572,MR0720931} the results extend ad verbatim to Lipschitz manifolds with boundary. To state the results, we need further notation. Write \(x_n\) for the inwards pointing normal coordinate near the boundary. For a \(k\)-form \(\omega\) we can near the boundary write 
\begin{equation}
\label{lkjnklajnda}
	\omega=\omega_1+\rd x_{ n}\wedge \omega_2,
\end{equation}
where \(\omega_1\) and \(\omega_2\) are defined near the boundary and take values in \(\wedge^kT^*\partial M\) and \(\wedge^{k-1}T^*\partial M\), respectively. In other words, \eqref{lkjnklajnda} uniquely decomposes $\omega$ into components \(\omega_1\) and \(\omega_2\) not containing $\rd x_{n}$. Following \cite[Section 2.7.1]{gilkeyinv} we can define the relative boundary condition \(B_r\) and the absolute boundary condition \(B_a\) by 
\[
	B_r \omega := \omega_1\lvert_{\partial M} =0 \text{ and } B_a \omega := \omega_2\lvert_{\partial M} =0,
\]

\begin{thm}
\label{thm:domain_of_absolute_bc_hodgedirac}
Let \(M\) be a compact Lipschitz manifold with boundary. The operator \(D_a := \rd_{\max}+\rd_{\max}^*\) is a self-adjoint realization of the Hodge-Dirac operator \(\slashed{D}=\rd+\delta\) with domain contained in the Sobolev space \(H^1(M;\wedge^*T^*M)\). In fact, 
\[ \Dom(D_a) := \setcond{u\in H^1(M;\wedge^*T^*M)}{B_au=0}, \]
and in the special case that \(M\) is a smooth manifold with smooth boundary then \(D_a\) is a Shapiro-Lopatinski elliptic boundary value problem.
\end{thm}

The reader can find more details about Shapiro-Lopatinski elliptic boundary value problems in \cite{MR1481215}. We refer the reader to \cite[Theorem 4.1.1]{leschbruening} for a proof of \autoref{thm:domain_of_absolute_bc_hodgedirac}. But to give the reader a feeling for the argument, we recall its salient features. The main idea is to go to the doubled Lipschitz manifold \(\tilde{M} := 2M\) and let \(\alpha:\tilde{M}\to \tilde{M}\) denote the flip map which is an involutive Lipeomorphism; in \cite{leschbruening} they remain within the smooth category. We equip \(\tilde{M}\) with the Riemannian structure making \(\alpha\) isometric. Now as in \cite{leschbruening}, the Hilbert complex \((\tilde{H}_\bullet,\tilde{\rd}_{\,\bullet})\) defined from the de Rham complex on \(\tilde{M}\) has only one ideal boundary condition (the minimal and maximal realization coincides). We decompose into the \(\pm1\)-eigenspaces for \(\alpha^*\) as 
\[ (\tilde{H}_\bullet,\tilde{\rd}_{\,\bullet})=\underbrace{(\tilde{H}_\bullet^a,\tilde{\rd}_{\,\bullet}^a)}_{=\ker(\alpha^*-1)}\oplus \underbrace{(\tilde{H}_{\,\bullet}^r,\tilde{\rd}_{\,\bullet}^r)}_{=\ker(\alpha^*+1)}. \]
As in \cite{leschbruening}, one proves that 
\[ (\tilde{H}_\bullet^a,\tilde{\rd}_{\,\bullet}^{\,a})|_M=(H_\bullet,\rd_{\bullet,{\max}}) \quad\mbox{and}\quad  (\tilde{H}_{\,\bullet}^r,\tilde{\rd}_{\, \bullet}^{\,r})|_M=(H_\bullet,\rd_{\bullet,{\min}}) . \]
From here, the proof proceeds as in \cite{leschbruening} to show that \(\rd_{\max}+\rd_{\max}^*\) is the realization defined from the boundary condition \(B_a\). We note here that the relative boundary condition \(B_r\) arises in the same way but from the minimal realization \(\rd_{\min}\).\\

By construction, we have that 
\[ D_a^2=\oplus_{k=0}^d \Delta_{k,a}, \]
so we can describe the form domain of \(\Delta_{k,a}\) rather easily using \autoref{thm:domain_of_absolute_bc_hodgedirac}. 

\begin{lemma}
\label{thm:q_ka}
The quadratic form \(q_{k,a}\) associated with \(\Delta_{k,a}\) takes the form 
\[
	q_{k,a}(\omega) := \int_{M}(|\rd_k\omega|^2+|\delta_{k-1}\omega|^2)\rd x,
\]
and its domain is given by 
\[ \Dom(q_{k,a}) := \Dom(\rd_{k,{\max}})\cap \Dom(\rd_{k-1,{\max}}^*)\equiv \setcond{u\in H^1(M;\wedge^kT^*M)}{B_au=0}.  \]
\end{lemma}

When \(M\) and its boundary are smooth, we can even describe the domain of \(\Delta_{k,a}\).  We define the boundary conditions \(	\mathcal{B}_a\)  for the Hodge Laplacian on \(k\)-forms \(\Delta_k\) as
\[\mathcal{B}_a\omega := (B_a\omega , B_a (\rd+\delta) \omega).\]
The identity \(D_a^2=\oplus_{k=0}^d \Delta_{k,a}\), \autoref{thm:domain_of_absolute_bc_hodgedirac} and elliptic regularity for Shapiro-Lopatinski elliptic boundary value problems implies the following.

\begin{thm}
\label{absolueteapp}
Let \(M\) be a smooth compact manifold with smooth boundary. The operator 
\[ \Delta_{k,a} := \rd_{k,{\max}}^*\rd_{k,{\max}}+\rd_{k-1,{\max}}\rd_{k-1,{\max}}^*, \] 
is a self-adjoint realization of the Hodge Laplacian on \(k\)-forms with domain contained in the Sobolev space \(H^2(M;\wedge^kT^*M)\). In fact, \(\Delta_{k,a}\) is a Shapiro-Lopatinski elliptic boundary value problem and
\[ \Dom(\Delta_{k,a}) := \setcond{u\in H^2(M;\wedge^kT^*M)}{\mathcal{B}_au=0}. \]
\end{thm}

We here impose the assumption that \(M\) is smooth to ensure that the Sobolev space \(H^2(M;\wedge^kT^*M)\) is well-defined and to be able to employ elliptic regularity. For instance, \autoref{absolueteapp} covers Euclidean domains with smooth boundary. For our considerations, we only need the quadratic form domain (as described in \autoref{thm:q_ka}) in the proofs of eigenvalue inequalities. We note also that \cite[Lemma 2.7.2]{gilkeyinv} ensures that \((\Delta_k, \mathcal{B}_a)\) is self-adjoint from first principles, and not only from that it coincides with our operator \(\Delta_{k,a}\).

Let us verify again that \(\Delta_{0,a}\) coincides with the Neumann realization of the Hodge Laplacian. In degree zero, \(B_a(\rd+\delta)\omega=B_a \rd_0 \omega =\partial_{x_{n}}\omega|_{\partial M}\). So by definition, for \(\omega\in C^\infty(M)=C^\infty(M,\wedge^0T^*M)\), 
\[\mathcal{B}_a \omega = (0, \partial_{x_{n}}\omega|_{\partial M}).\]
In particular, in degree zero, 
\[ \mathcal{B}_a \omega =0\Leftrightarrow   \partial_{x_{n}}\omega|_{\partial M}=0,  \]
which gives the desired  Neumann boundary condition. If one wants to check that \(\Delta_{d,a}\) coincides with the Dirichlet realization, one computes that for \(\omega=\omega_0\rd V\in C^\infty(M,\wedge^dT^*M)\), with \(\omega_0\in C^\infty(M)\) and \(\rd V=\star(1)\) the Riemannian volume form, that
\[\mathcal{B}_a \omega = ((\rd x_{n}\neg \omega)|_{\partial M}, 0), \quad\mbox{so}\quad  \mathcal{B}_a \omega=0\Leftrightarrow   \omega_0|_{\partial M}=0.  \]
Here \(\rd x_{n}\neg \omega\) denotes the contraction of \(\omega\) along the normal covector \(\rd x_{n}\). Moreover, by \cite[Lemma 2.7.1]{gilkeyinv} the Hodge star \(\star\) implements an identification of \(\oplus_{k=0}^d\Delta_{k,a}\) with the corresponding relative/minimal realization \(\oplus_{k=0}^d\Delta_{k,r}\) at the cost of flipping degree \(k\) forms to degree \(d-k\)-forms. Indeed, \(B_a \omega= 0 \Leftrightarrow B_r \star \omega =0\). 

\begin{remark}
\label{lkjnkjnkjanda}
The content of \cite[Theorem 4.1.2]{leschbruening} is precisely that 
\[
	\ker(\Delta_{k,a})\cong H^k(M;\C).
\]
This is the motivation for using the term absolute boundary conditions, since the associated space of harmonic forms realizes the absolute cohomology groups $H^*(M;\C)$. In particular, the definition of Euler characteristic (see Equation \eqref{eulerdefe}) and the Hodge decomposition \eqref{hodgedecomp} implies that
\[ \chi(L^2(M;\wedge^\bullet T^*M),\rd_{\bullet, {\max}})=\chi(M). \]
So, for any \(\lambda\geq 0\), \autoref{thm:complex_alternating_sum} implies
\[ \sum_{k=0}^{d} (-1)^k N(\Delta_{k,a}, \lambda)= \chi(M). \]
If we use the minimal/relative realization, we instead have the equalities
\[ \ker(\Delta_{k,r})\cong H^k(M,\partial M;\C), \quad\mbox{and}\quad \chi(L^2(M;\wedge^\bullet T^*M),\rd_{\bullet, {\min}})=\chi(M,\partial M), \]
where \(H^k(M,\partial M;\C)\) denotes the \(k\):th relative cohomology with respect to the boundary inclusion \(\partial M\hookrightarrow M\) and \(\chi(M,\partial M) := \sum_{k=0}^{d} (-1)^k \dim_\C H^k(M,\partial M;\C)\). This is the motivation for using the term relative boundary conditions, since the associated space of harmonic forms realizes the relative cohomology groups $H^*(M,\partial M;\C)$. 
\end{remark}

\subsection{Some computations on domains in \texorpdfstring{\(\R^d\)}{ Rd}}
\label{subsec:computation_in_Rd}

We are primarily interested in \(\Omega\subseteq \R^d\) being a domain. In this case, we use the standard basis for the exterior algebra. That is, we construct an ON-basis \(\rd x_I\) for the \(k\)-forms labeled by ordered sets \(I=\set{i_1<i_2<\cdots<i_k}\), where \(i_j\in \set{1,\ldots, d}\), as
\[ \rd x_I := \rd x_{i_1}\wedge \rd x_{i_2}\wedge \cdots\wedge \rd x_{i_k}. \]
In particular, we see that 
\begin{equation}
	\label{eq:deimdiniand}
	\mathrm{rk} \wedge^k T^*\Omega=\dim \wedge^k \C^d=\begin{pmatrix} d\\ k\end{pmatrix}.
\end{equation}
For instance, \(\rd x_1\wedge\rd x_2\wedge \cdots \wedge \rd x_d\) is the basis element of choice for \(\wedge^d T^*\Omega\). In these bases, for a function \(f\in C^\infty(\Omega)\) we have that 
\[
	\slashed{D}(f\rd x_I)=
\sum_{j\notin I}\partial_{x_j}f\rd x_j\wedge \rd x_I-\sum_{j\in I}\sign(j,I)\partial_{x_j}f\rd x_{I\setminus \{j\}},
\]
where \(\sign(j,I)\in \set{-1,1}\) is determined by \(\rd x_j\wedge \rd x_{I\setminus \{j\}}=\sign(j,I)\rd x_I\). We then have
\begin{equation}
\label{eq:dslahsedcoordlaplacian}
\slashed{D}^2(f\rd x_I)=(\Delta f)\rd x_I.
\end{equation}
In particular, \(\Delta_{k,a}\) is a realization of the scalar Laplacian on each of the basis vectors of \(\wedge^k \mathbb{C}^d\).

\section{Spectral properties of the de Rham complex}
\label{sec:SpectralPropdeRham}
In this section we use the notions and results introduced in the previous sections to obtain estimates using the ideas from \cite{rohlfr} applied to the de Rham complex. In the next section they are used to prove Rohleder's inequality \eqref{eq:rohlineq}. We will henceforth only consider a domain \(\Omega\subseteq \R^d\) which is bounded, connected and has Lipschitz boundary. We provide a higher-dimensional analog to estimates appearing in \cite{rohlfr,RohlederCurlCurl}. We believe they are of significant interest for future considerations improving estimates between eigenvalues of Dirichlet and Neumann Laplacians, for instance the conjectural bound \eqref{eq:conjecture}.
 
\begin{lemma}
\label{lemma:rohledergeneralization}
Let \(\Omega\subset \mathbb{R}^d\) be a connected and bounded domain with Lipschitz boundary with \(d\geq 2\) and  \((\Delta_{k,a})_{k=0,\dots,d}\) be the corresponding Hodge Laplacians on \(k\)-forms with absolute boundary conditions. For \(\lambda\geq 0\),
\[ d ~ N(\Delta_{d,a},\lambda) + m(\Delta_{d,a}, \lambda) \leq N(\Delta_{d-1,a},\lambda). \]
\end{lemma}

\begin{proof}
Let \(N := N(\Delta_{d,a}, \lambda)\) and select \(N\) orthonormal eigenfunctions \(f_1 dV,\ldots, f_N dV\)  of \(\Delta_{d,a}\) with eigenvalues less or equal to \(\lambda\). Consider the \(dN\)-dimensional space
\[V_1 := \Span\setcond{f_j\widehat{\rd x_l}}{j=1,\ldots, N, \; l=1,\ldots, d} \subseteq H_0^1(\Omega;\wedge^{d-1}T^*\Omega)\] 
where \(\widehat{\rd x_l} :=  \rd x_1\wedge\cdots \wedge \rd x_{l-1}\wedge \rd x_{l+1}\wedge\cdots \rd x_d\) denotes the standard ON-basis for \(\wedge^{d-1}\mathbb{C}^d\). We can estimate \(q_{d-1, a}(u)\leq \lambda\norm{u}_{L^2(\Omega;\wedge^{d-1}T^*\Omega)}^2\) for \(u\in V_1\).

Next, we will consider the space 
\[
	V_2 := \delta_{d-1}\ker(\Delta_{d,a} - \lambda)
\]
for which \(\dim V_2 = \dim \delta_{d-1}(\ker (\Delta_{d,a} - \lambda)) = \dim (\ker (\Delta_{d,a} - \lambda)) = m(\Delta_{d, a}, \lambda)\). We will show that \(V_1\cap V_2 = 0\), that is, if \(\delta_{d-1} g\in V_1\) for some \(g\in \ker (\Delta_{d,a} - \lambda)\) then \(g\) is identically zero (this proof is similar to \cite[Lemma 3.2.36]{SpectralGeometryBook}). We denote by \(\gamma_0\) the trace operator. 
Since \(g\in \ker(\Delta_{d,a} -\lambda)\) we have that \(g\in H^1_0(\Omega; \wedge^d T^\ast \Omega)\), which implies that \(\gamma_0 g =0\). Moreover, since \(\delta_{d-1} g \in V_1 \subseteq H^1_0(\Omega;\wedge^{d-1}T^*\Omega)\) we have that \(\gamma_0 \delta_{d-1} g = 0\). This means that we can extend \(g\) by zero to \(\tilde{g}\in H^1(\mathbb{R}^d; \wedge^d \R^d)\). For any \(v\in C_c^{\infty}(\mathbb{R}^d; \wedge^d \R^d)\)
\begin{align*}
\braket{\delta_{d-1} \tilde{g}, \delta_{d-1} v}_{L^2(\mathbb{R}^d; \wedge^{d-1} \mathbb{R}^d)}=~& \braket{\delta_{d-1} g, \delta_{d-1} v}_{L^2(\Omega; \wedge^{d-1} T^*\Omega)} \\
=~& \braket{\Delta_d g,  v}_{L^2(\Omega; \wedge^d T^*\Omega)}  \\
=~& \lambda  \braket{ g,  v}_{L^2(\Omega; \wedge^d T^*\Omega)} \\
=~& \lambda  \braket{ g,  v}_{L^2(\mathbb{R}^d; \wedge^d \mathbb{R}^d)},
\end{align*}
where the boundary term is zero because  \(\gamma_0 \delta_{d-1} g = 0\). Therefore, \(\tilde{g} \in H^1(\mathbb{R}^d; \wedge^d \C^d)\) is a solution of \(\Delta_d \tilde{g} = \lambda \tilde{g}\) in the weak sense on \(\mathbb{R}^d\). By elliptic regularity, we get that \(\tilde{g}\) is real-analytic. Since \(\tilde{g}|_{\mathbb{R}^d \setminus \Omega} \equiv 0 \), then unique continuation results imply that \(\tilde{g}\equiv 0\). This ensures that the space \(V := V_1+V_2\) has
\begin{align}
	\label{eq:dim_dirichlet_and_multiplicity}
	\dim(V) =~& d~N + m(\Delta_{d,a}, \lambda). 
\end{align}
Lastly, for \(u\in V\) of the form \(u=v_1+v_2\) with \(v_1\in V_1\subseteq H^1_0(\Omega,\wedge^{d-1}T^*\Omega)\) and \(v_2\in V_2\subseteq \ker(\Delta_{d-1,a}-\lambda)\) we see that
\begin{equation*}
	\begin{split}
	q_{d-1,a}(u) &= \norm{\rd_{d-1}v_1 + \rd_{d-1}v_2}^2_{L^2(\Omega;\wedge^{d}T^*\Omega)}+ \norm{\delta_{d-2}v_1}^2_{L^2(\Omega;\wedge^{d-2}T^*\Omega)} \\
				&= q_{d-1,a}(v_1) + \norm{\rd_{d-1}v_2}^2_{L^2(\Omega;\wedge^{d}T^*\Omega)} + 2\real \braket{\rd_{d-1}v_1, \rd_{d-1}v_2}_{L^2(\Omega;\wedge^{d}T^*\Omega)} \\
				&= q_{d-1,a}(v_1) + \lambda \norm{v_2}^2_{L^2(\Omega;\wedge^{d-1}T^*\Omega)} + 2 \lambda \real\braket{v_1, v_2}_{L^2(\Omega;\wedge^{d}T^*\Omega)} \\
				&\leq \lambda \norm{u}_{L^2(\Omega;\wedge^{d-1}T^*\Omega)}^2
\end{split}
\end{equation*}
The result follows from Glazman's lemma using the linear subspace \(V\).
\end{proof}

\section{Comparing forms to vectors in two and three dimensions}
\label{sec:Rohleder_language}

The aim of this section is to rewrite the results of \autoref{sec:SpectralPropdeRham} in vector operators in two and three dimensions to see how our results compare to those in \cite{rohlfr, RohlederCurlCurl,kennedy2024hotspotsconjecturehigher}. 

\subsection{Inequalities in dimension 2}
Let \(\Omega\subset \mathbb{R}^2\) be a connected and bounded domain with Lipschitz boundary. In two dimensions, the exterior differential on \(1\)-forms introduced in \eqref{eq:exterior_differential} acts as
\[ \rd_1 (u_1 \rd x_1 + u_2 \rd x_2) =  \big(-\partial_{x_2}u_1 + \partial_{x_1}u_2\big) \rd x_1\wedge \rd x_2, \]
and the formal adjoint of the exterior differential on \(0\)-forms introduced in \eqref{eq:delta_def} as
\[
	\delta_0 (u_1 \rd x_1 + u_2 \rd x_2)=  \partial_{x_1}u_1 + \partial_{x_2}u_2,
\]
which can be identified with the differential expressions \(\omega(u) := \partial_{x_1}u_2 - \partial_{x_2}u_1\) and \(\divergence u\) introduced in \cite{rohlfr}. In other words, the form \(q_{1,a}\) associated with \(\Delta_{1,a}\) with domain \(\Dom(q_1,a) = \Dom (\rd_{1,{\max}})\cap \Dom(\delta_{0,{\max}})\) is exactly the same as
\[
	\mathfrak{a}[u,v] = \int_{\Omega} (\divergence u \, \overline{\divergence v} + \omega(u)\overline{\omega (v)})\, \rd x,
\]
with
\[ \Dom \mathfrak{a} =\setcond{u\in L^2(\Omega)^2}{\divergence u, \omega(u) \in L^2(\Omega), \braket{u|_{\partial \Omega} , \nu}=0}, \] 
where \(\nu\) the unit normal vector, introduced in \cite[Section 3]{rohlfr}. This means that the operator \(A\) introduced in \cite[Proposition 3.1]{rohlfr} coincides with \(\Delta_{1,a}\). Next, we present \cite[Theorem 4.1]{rohlfr} and give an analogous proof using \autoref{lemma:rohledergeneralization}.

\begin{prop}
	\label{prop:Rohleder2D}
	Let \(\Omega\subset \mathbb{R}^2\) be a connected and bounded domain with Lipschitz boundary. Let \(j=1,2,\ldots\) then
	\begin{equation}
		\label{eq:Rohleder2D}
		\lambda_{j+\chi(\Omega)+m(\Delta_D, \lambda_j(\Delta_D))}(\Delta_N) \leq \lambda_j (\Delta_D) .
	\end{equation}
\end{prop}

\begin{proof}
	We let \(j\in \mathbb{N}\), and fix \(\lambda := \lambda_j(\Delta_{2,a})\) as well as \(m := m(\Delta_{2,a}, \lambda_j(\Delta_{2,a}))\). By \autoref{lemma:rohledergeneralization} we know that 
	\begin{equation}\label{eq:lemma4.2_in_2D}
		2N(\Delta_{2,a}, \lambda) + m\leq N(\Delta_{1,a}, \lambda),
	\end{equation}
and by \autoref{thm:complex_alternating_sum} we have
	\begin{equation}
	\label{eq:alternating_sum2D}
		N(\Delta_{0,a}, \lambda) - N(\Delta_{1,a}, \lambda)+N(\Delta_{2,a}, \lambda) = \chi(\Omega).
	\end{equation}
Combining \eqref{eq:lemma4.2_in_2D} with  \eqref{eq:alternating_sum2D} 
	\begin{equation*}
		\begin{split}
			N(\Delta_{0,a}, \lambda)  \geq&  \chi(\Omega)  + N(\Delta_{2,a}, \lambda) + m
		\end{split}
	\end{equation*}
which gives \eqref{eq:Rohleder2D} since \(\Delta_N = \Delta_{0,a}\) and \(\Delta_D = \Delta_{2,a}\) (see \autoref{sec:derahm}).
\end{proof}

\begin{cor}\cite[Theorem 4.1]{rohlfr}
\label{cor:Rohleder2D}
	Let \(\Omega\subset \mathbb{R}^2\) be a simply connected and bounded domain with Lipschitz boundary. Then
	\[
		\lambda_{j+2}(\Delta_N) \leq \lambda_j (\Delta_D)
	\]
	for all \(j=1,2,3,\ldots\).
\end{cor}

\begin{proof}
	If \(\Omega\subset \mathbb{R}^2\) is simply connected, then \(\chi(\Omega)=1\), so the result follows from \autoref{prop:Rohleder2D} and \(m(\Delta_D, \lambda_j(\Delta_D))\geq 1\). 
\end{proof}

\begin{remark}
	If \(\Omega\subseteq \R^2\) is non-simply connected, \autoref{prop:Rohleder2D} contains no new information beyond Friedlander's inequality \eqref{eq:fried}. In fact, if \(\Omega\) has $g=1$ holes we retrieve Friedlander's inequality \eqref{eq:fried}, for \(g=2\) holes then \(\chi(\Omega)=-1\) and \autoref{prop:Rohleder2D} gives the same bound as simply applying Glazman's lemma to the fact that the form domain of \(\Delta_D\) is contained in the form domain of \(\Delta_N\). If \(\Omega\) has $g>2$ holes, \autoref{prop:Rohleder2D} gives a worse bound than variational principles.
\end{remark}

\begin{remark}
	In the previous corollary we used the fact that \(m(\Delta_D, \lambda_j(\Delta_D))\geq 1\). Note that keeping the multiplicity term in \eqref{eq:Rohleder2D} will give 
	\[
		\lambda_{j+1+m(\Delta_D, \lambda_j(\Delta_D))}(\Delta_N) \leq \lambda_j (\Delta_D). 
	\]
	This inequality can be observed in the case of the unit disc where \(m(\Delta_D, \lambda_2(\Delta_D)) = 2 \), i.e. \(\lambda_2(\Delta_D) = \lambda_3(\Delta_D)\). For the disc we know that \(\lambda_5(\Delta_N) <\lambda_3(\Delta_D)\) but \(\lambda_6(\Delta_N)= \lambda_3(\Delta_D)\), where this equality comes from the fact that the zeros \(j_{m,n}\) of the \(m\):th Bessel function \(J_m(r)\) and the positive zeros \(j'_{m,n}\) of the derivative \(J'_m(r)\) fulfill \(j_{1,n} =j'_{0,n+1}\) for \(n \in \mathbb{N}\).
\end{remark}

\begin{remark}\label{remark:strictineq2D}
	Rohleder was able to obtain strict inequality in \eqref{eq:Rohleder2D} for simply connected domains if \(\lambda_k(\Delta_D)\) is a simple eigenvalue or \(\partial  \Omega\) contains a straight line segment. We refer to \cite[Theorem 4.1]{rohlfr} for the proof. 
\end{remark}

\subsection{Rohleder's bound on eigenvalues for the curl curl operator}

Let \(\Omega\subset \mathbb{R}^d\) be a connected and bounded domain with Lipschitz boundary. We can define a positive, self-adjoint operator 
\[ \mathfrak{C} := \rd_{d-2,{\max}}\rd_{d-2,{\max}}^*, \] 
which is densely defined on the Hilbert subspace 
\[ \ker(\rd_{d-1,{\max}})\subseteq L^2(\Omega,\wedge^{d-1}T^*\Omega). \]
The form associated with the operator \(\mathfrak{C}\) takes the form
\[ q_\mathfrak{C}(u) := \int_\Omega |\delta_{d-2}u|^2\rd x, \]
that by \autoref{thm:q_ka} has the domain
\begin{align*}
\Dom(q_\mathfrak{C})\equiv& \ker(\rd_{d-1,{\max}})\cap \Dom(\rd_{d-2,{\max}}^*)\\
=~&\setcond{u\in \ker(\rd_{d-1,{\max}})\cap H^1(\Omega;\wedge^{d-1}T^*\Omega)}{B_au=0}.
\end{align*}
For a \((d-1)\)-form \(u\), we can near the boundary write \(u=u_0\rd V_\partial+\rd x_n\wedge u_2\) where \(u_0\) is a scalar function, \(\rd V_\partial\) the volume form on \(\partial \Omega\) induced from the Euclidean metric, and \(u_2\) is a section to \(\wedge^{d-2}T^*\partial \Omega\). In particular, if \(u\in \ker(\rd_{d-1,{\max}})\cap H^1(\Omega;\wedge^{d-1}T^*\Omega)\) then \(B_au=0\) if and only if \(u_2|_{\partial \Omega}=0\). In analogy with \cite{RohlederCurlCurl}, we call \(\mathfrak{C}\) the curl curl operator.

\begin{prop}
	\label{prop:RohlederCurl}
	Let \(\Omega\subset \mathbb{R}^d\) be a connected and bounded domain with Lipschitz boundary and write \(\mathfrak{C}\) for its curl curl operator. Let \(j=1,2,\ldots\) then
	\begin{equation}
		\label{eq:RohlederdD}
		\lambda_{(d-1)j+m(\Delta_D, \lambda_j(\Delta_D))}(\mathfrak{C}) \leq \lambda_j (\Delta_D) .
	\end{equation}
\end{prop}

\begin{proof}
	By the definition of \(\mathfrak{C}\) and the Hodge decomposition, we have for \(\lambda\geq 0\) that 
	\[ N(\mathfrak{C},\lambda)=\dim (\ker(\Delta_{d-1,a}))+N(\rd_{d-2,{\max}}\rd^*_{d-2,{\max}};(0,\lambda]). \]
	In particular, using Equation \eqref{eq:hodgrealref}
	\begin{align*}
		N(\mathfrak{C},\lambda)=~&N(\Delta_{d-1,a},\lambda)-N(\rd^*_{d-1,{\max}}\rd_{d-1,{\max}};(0,\lambda])\\
		=~&N(\Delta_{d-1,a},\lambda)-N(\Delta_D,\lambda).
	\end{align*}
	From \autoref{lemma:rohledergeneralization}, we see that
	\[ N(\mathfrak{C},\lambda)\geq (d-1) ~ N(\Delta_D,\lambda) + m(\Delta_D, \lambda), \]
	and the proof is complete.
\end{proof}

In three dimensions, the exterior codifferential \(\delta_1\) on \(2\)-forms acts as
\begin{align*}
	\delta_1 (u_1 \rd x_2\wedge \rd x_3 + &u_2 \rd x_3\wedge \rd x_1+u_3\rd x_1\wedge \rd x_2) \\
=~&  (\partial_2u_3-\partial_3u_2)\rd x_1+(\partial_1u_3-\partial_3u_1)\rd x_2+(\partial_2u_1-\partial_1u_2)\rd x_3,
\end{align*}
so up to the Hodge star we can identify \(\delta_1\) with the curl operator in three dimensions. A similar computation shows that \(\rd_2\) can be identified with the divergence of vector fields. We note the discussion above shows that a \(2\)-form \(u\) belongs to \(\Dom(q_\mathfrak{C})\) if and only if \(\rd_2u=0\) in distributional sense and \(\star u\) restricts to the zero form on \(\partial \Omega\). If we identify \(2\)-forms with vector fields via the Hodge star, this means that \(u\) belongs to \(\Dom(q_\mathfrak{C})\) if and only if \(\mathrm{div}(u)=0\) in distributional sense and \(u\times \nu=0\) on \(\partial \Omega\). We see that in dimension \(3\), \(\mathfrak{C}\) coincides with the curl curl operator defined in \cite{RohlederCurlCurl} and in the notation of \cite{RohlederCurlCurl}, \(\lambda_j(\mathfrak{C})=\alpha_j\). In particular, \autoref{prop:RohlederCurl} extends \cite[Theorem 1.1]{RohlederCurlCurl} from dimension three to arbitrary dimension.

\begin{remark}
	In \cite[Theorem 1.1]{RohlederCurlCurl} strict inequality in \eqref{eq:RohlederdD} is attained when \(\Omega\) is a polyhedron or \(\lambda_k(\Delta_D)\) is a simple eigenvalue. An analogous proof could be carried on to obtain strict inequality between \(\mathfrak{C}\) and \(\Delta_D\) for \(d\geq 2\). 
\end{remark}

\section*{Acknowledgements}
We thank Jonathan Rohleder for his comments on an early version of the paper, as well as comments on the published version. We also thank Nikolay Filonov for a counterexample to an erroneous claim in a previous version. We thank Mikael Persson Sundqvist for the useful discussions. We are most grateful to Matthias Lesch for his careful explanations and helpful remarks on \cite{leschbruening}. G.M. would like to thank Daniel Sahlberg for providing numerical examples concerning domains in two dimensions. 

\bibliography{mybib.bib}

@article {FGMFriedlanderreftoJST,
    AUTHOR = {Fries, Magnus and Goffeng, Magnus and Miranda, Germ\'an},
     TITLE = {Reproving {F}riedlander's inequality with the de {R}ham
              complex},
   JOURNAL = {J. Spectr. Theory},
  FJOURNAL = {Journal of Spectral Theory},
    VOLUME = {15},
      YEAR = {2025},
    NUMBER = {4},
     PAGES = {1593--1613},
      ISSN = {1664-039X,1664-0403},
   MRCLASS = {35P05},
  MRNUMBER = {4980388},
       DOI = {10.4171/jst/582},
       URL = {https://doi.org/10.4171/jst/582},
}

@book {warnerbook,
    AUTHOR = {Warner, Frank W.},
     TITLE = {Foundations of differentiable manifolds and {L}ie groups},
    SERIES = {Graduate Texts in Mathematics},
    VOLUME = {94},
      NOTE = {Corrected reprint of the 1971 edition},
 PUBLISHER = {Springer-Verlag, New York-Berlin},
      YEAR = {1983},
     PAGES = {ix+272},
      ISBN = {0-387-90894-3},
   MRCLASS = {58-01 (22-01 53-01 57R99)},
  MRNUMBER = {722297},
}

@book {guilleminhaine,
    AUTHOR = {Guillemin, Victor and Haine, Peter},
     TITLE = {Differential forms},
 PUBLISHER = {World Scientific Publishing Co. Pte. Ltd., Hackensack, NJ},
      YEAR = {2019},
     PAGES = {xvi+256},
      ISBN = {978-981-3272-77-4},
   MRCLASS = {58A10 (58-01 58A12)},
  MRNUMBER = {3931306},
       DOI = {10.1142/11058},
       URL = {https://doi.org/10.1142/11058},
}

@incollection {MR1481215,
    AUTHOR = {Agranovich, M. S.},
     TITLE = {Elliptic boundary problems},
 BOOKTITLE = {Partial differential equations, {IX}},
    SERIES = {Encyclopaedia Math. Sci.},
    VOLUME = {79},
     PAGES = {1--144, 275--281},
      NOTE = {Translated from the Russian by the author},
 PUBLISHER = {Springer, Berlin},
      YEAR = {1997},
   MRCLASS = {35J40 (35J55 35P20 47F05 58G03)},
  MRNUMBER = {1481215},
MRREVIEWER = {Norbert Weck},
       DOI = {10.1007/978-3-662-06721-5\_1},
       URL = {https://doi.org/10.1007/978-3-662-06721-5_1},
}

@article{leschbruening,
	title = {Hilbert Complexes},
	author = {Br{\"u}ning, J. and Lesch, M.},
	year = {1992},
	month = aug,
	journal = {Journal of Functional Analysis},
	volume = {108},
	number = {1},
	pages = {88--132},
	issn = {0022-1236},
	doi = {10.1016/0022-1236(92)90147-B},
	urldate = {2024-05-29}
}

@article {rohlfr,
    AUTHOR = {Rohleder, Jonathan},
     TITLE = {Inequalities between {N}eumann and {D}irichlet {L}aplacian
              eigenvalues on planar domains},
   JOURNAL = {Math. Ann.},
  FJOURNAL = {Mathematische Annalen},
    VOLUME = {392},
      YEAR = {2025},
    NUMBER = {4},
     PAGES = {5553--5571},
      ISSN = {0025-5831,1432-1807},
   MRCLASS = {31A10},
  MRNUMBER = {4958511},
       DOI = {10.1007/s00208-025-03216-4},
       URL = {https://doi.org/10.1007/s00208-025-03216-4},
}

@article {Polya,
	AUTHOR = {P{\'{o}}lya, G.},
	TITLE = {Remarks on the foregoing paper},
	JOURNAL = {J. Math. Physics},
	FJOURNAL = {Journal of Mathematics and Physics},
	VOLUME = {31},
	YEAR = {1952},
	PAGES = {55--57},
	ISSN = {0097-1421},
	MRCLASS = {36.0X},
	MRNUMBER = {47237},
	MRREVIEWER = {P.\ Funk},
}

@article {Payne,
	AUTHOR = {Payne, L. E.},
	TITLE = {Inequalities for eigenvalues of membranes and plates},
	JOURNAL = {J. Rational Mech. Anal.},
	FJOURNAL = {Journal of Rational Mechanics and Analysis},
	VOLUME = {4},
	YEAR = {1955},
	PAGES = {517--529},
	ISSN = {1943-5282,1943-5290},
	MRCLASS = {35.0X},
	MRNUMBER = {70834},
	MRREVIEWER = {G.\ Szeg\"{o}},
	DOI = {10.1512/iumj.1955.4.54016},
	URL = {https://doi.org/10.1512/iumj.1955.4.54016},
}

@article {LevineWeinberger,
	AUTHOR = {Levine, Howard A. and Weinberger, Hans F.},
	TITLE = {Inequalities between {D}irichlet and {N}eumann eigenvalues},
	JOURNAL = {Arch. Rational Mech. Anal.},
	FJOURNAL = {Archive for Rational Mechanics and Analysis},
	VOLUME = {94},
	YEAR = {1986},
	NUMBER = {3},
	PAGES = {193--208},
	ISSN = {0003-9527},
	MRCLASS = {35P15 (35J05)},
	MRNUMBER = {846060},
	MRREVIEWER = {Gerd\ Grubb},
	DOI = {10.1007/BF00279862},
	URL = {https://doi.org/10.1007/BF00279862},
}

@article {Friedlander,
	AUTHOR = {Friedlander, Leonid},
	TITLE = {Some inequalities between {D}irichlet and {N}eumann
		eigenvalues},
	JOURNAL = {Arch. Rational Mech. Anal.},
	FJOURNAL = {Archive for Rational Mechanics and Analysis},
	VOLUME = {116},
	YEAR = {1991},
	NUMBER = {2},
	PAGES = {153--160},
	ISSN = {0003-9527},
	MRCLASS = {35P15 (47F05 58G25)},
	MRNUMBER = {1143438},
	DOI = {10.1007/BF00375590},
	URL = {https://doi.org/10.1007/BF00375590},
}

@article {Filonov,
	AUTHOR = {Filonov, N.},
	TITLE = {On an inequality for the eigenvalues of the {D}irichlet and
		{N}eumann problems for the {L}aplace operator},
	JOURNAL = {Algebra i Analiz},
	FJOURNAL = {Rossi\u{\i}skaya Akademiya Nauk. Algebra i Analiz},
	VOLUME = {16},
	YEAR = {2004},
	NUMBER = {2},
	PAGES = {172--176},
	ISSN = {0234-0852},
	MRCLASS = {35P15 (35J05 47F05)},
	MRNUMBER = {2068346},
	MRREVIEWER = {Leonid\ Friedlander},
	DOI = {10.1090/S1061-0022-05-00857-5},
	URL = {https://doi.org/10.1090/S1061-0022-05-00857-5},
}

@article {LaptevFrank,
	AUTHOR = {Frank, Rupert L. and Laptev, Ari},
	volume = {2010},
	TITLE = {Inequalities between {D}irichlet and {N}eumann eigenvalues on
		the {H}eisenberg group},
	JOURNAL = {Int. Math. Res. Not. IMRN},
	FJOURNAL = {International Mathematics Research Notices. IMRN},
	YEAR = {2010},
	NUMBER = {15},
	PAGES = {2889--2902},
	ISSN = {1073-7928,1687-0247},
	MRCLASS = {58J50 (22E30 43A80)},
	MRNUMBER = {2673713},
	MRREVIEWER = {Jingzhi\ Tie},
	DOI = {10.1093/imrn/rnp230},
	URL = {https://doi.org/10.1093/imrn/rnp230},
}

@article {Mazzeo,
	AUTHOR = {Mazzeo, Rafe},
	TITLE = {Remarks on a paper of {L}. {F}riedlander concerning
		inequalities between {N}eumann and {D}irichlet eigenvalues:
		``{S}ome inequalities between {D}irichlet and {N}eumann
		eigenvalues'' [{A}rch. {R}ational {M}ech. {A}nal. {\bf 116}
		(1991), no. 2, 153--160; {MR}1143438 (93h:35146)]},
	JOURNAL = {Internat. Math. Res. Notices},
	FJOURNAL = {International Mathematics Research Notices},
	YEAR = {1991},
	NUMBER = {4},
	PAGES = {41--48},
	ISSN = {1073-7928,1687-0247},
	MRCLASS = {35P15 (47F05 58G25)},
	MRNUMBER = {1121164},
	MRREVIEWER = {Carolyn\ Gordon},
	DOI = {10.1155/S1073792891000065},
	URL = {https://doi.org/10.1155/S1073792891000065},
}

@article {RohlederSchrodinger,
	AUTHOR = {Rohleder, Jonathan},
	TITLE = {Inequalities between {N}eumann and {D}irichlet eigenvalues of
		{S}chr\"{o}dinger operators},
	JOURNAL = {J. Spectr. Theory},
	FJOURNAL = {Journal of Spectral Theory},
	VOLUME = {11},
	YEAR = {2021},
	NUMBER = {3},
	PAGES = {915--933},
	ISSN = {1664-039X,1664-0403},
	MRCLASS = {35P15 (35J10)},
	MRNUMBER = {4322026},
	MRREVIEWER = {Rupert\ L.\ Frank},
	DOI = {10.4171/jst/361},
	URL = {https://doi.org/10.4171/jst/361},
}

@book {gilkeyinv,
    AUTHOR = {Gilkey, Peter B.},
     TITLE = {Invariance theory, the heat equation, and the
              {A}tiyah-{S}inger index theorem},
    SERIES = {Studies in Advanced Mathematics},
   EDITION = {Second},
 PUBLISHER = {CRC Press, Boca Raton, FL},
      YEAR = {1995},
     PAGES = {x+516},
      ISBN = {0-8493-7874-4},
   MRCLASS = {58Gxx (58G10 58G11)},
  MRNUMBER = {1396308},
MRREVIEWER = {Matthias Lesch},
}

@book {SpectralGeometryBook,
	AUTHOR = {Levitin, Michael and Mangoubi, Dan and Polterovich, Iosif},
	TITLE = {Topics in spectral geometry},
	SERIES = {Graduate Studies in Mathematics},
	VOLUME = {237},
	PUBLISHER = {American Mathematical Society, Providence, RI},
	YEAR = {2023},
	PAGES = {xviii+325},
	ISBN = {[9781470475253]; [9781470475482]; [9781470475499]},
	MRCLASS = {35-01 (35Pxx 47A75 58J50 58J53 65N30)},
	MRNUMBER = {4655924},
	DOI = {10.1090/gsm/237},
	URL = {https://doi.org/10.1090/gsm/237},
}

@article {RohlederCurlCurl,
    AUTHOR = {Rohleder, Jonathan},
     TITLE = {Curl curl versus {D}irichlet {L}aplacian eigenvalues},
   JOURNAL = {Bull. Lond. Math. Soc.},
  FJOURNAL = {Bulletin of the London Mathematical Society},
    VOLUME = {57},
      YEAR = {2025},
    NUMBER = {9},
     PAGES = {2738--2747},
      ISSN = {0024-6093,1469-2120},
   MRCLASS = {35P15 (35Q61)},
  MRNUMBER = {4956743},
       DOI = {10.1112/blms.70121},
       URL = {https://doi.org/10.1112/blms.70121},
}

@misc{cox2019isoperimetricrelationsdirichletneumann,
	title = {Isoperimetric relations between {Dirichlet} and {Neumann} eigenvalues},
	url = {http://arxiv.org/abs/1906.10061},
	doi = {10.48550/arXiv.1906.10061},
	urldate = {2024-12-03},
	publisher = {arXiv},
	author = {Cox, Graham and MacLachlan, Scott Scott and Steeves, Luke},
	month = jun,
	year = {2019},
	note = {arXiv:1906.10061},
}

@article {BenguriaLevitinParnovski,
    AUTHOR = {Benguria, Rafael and Levitin, Michael and Parnovski, Leonid},
     TITLE = {Fourier transform, null variety, and {L}aplacian's
              eigenvalues},
   JOURNAL = {J. Funct. Anal.},
  FJOURNAL = {Journal of Functional Analysis},
    VOLUME = {257},
      YEAR = {2009},
    NUMBER = {7},
     PAGES = {2088--2123},
      ISSN = {0022-1236,1096-0783},
   MRCLASS = {35P15 (42B99)},
  MRNUMBER = {2548031},
MRREVIEWER = {M.\ S.\ Agranovich},
       DOI = {10.1016/j.jfa.2009.06.022},
       URL = {https://doi.org/10.1016/j.jfa.2009.06.022},
}

@article {Aviles,
    AUTHOR = {Aviles, Patricio},
     TITLE = {Symmetry theorems related to {P}ompeiu's problem},
   JOURNAL = {Amer. J. Math.},
  FJOURNAL = {American Journal of Mathematics},
    VOLUME = {108},
      YEAR = {1986},
    NUMBER = {5},
     PAGES = {1023--1036},
      ISSN = {0002-9327,1080-6377},
   MRCLASS = {35J65 (35P05 43A85 47H15 58G20)},
  MRNUMBER = {859768},
MRREVIEWER = {Carlos\ A.\ Berenstein},
       DOI = {10.2307/2374594},
       URL = {https://doi.org/10.2307/2374594},
}

@misc{frank2024inequalitiesdirichletneumanneigenvalues,
	title = {Inequalities between {Dirichlet} and {Neumann} eigenvalues on {Carnot} groups},
	url = {http://arxiv.org/abs/2411.11168},
	doi = {10.48550/arXiv.2411.11168},
	urldate = {2024-12-03},
	publisher = {arXiv},
	author = {Frank, Rupert L. and Helffer, Bernard and Laptev, Ari},
	month = nov,
	year = {2024},
	note = {arXiv:2411.11168},
}

@incollection {MR0799572,
    AUTHOR = {Hilsum, Michel},
     TITLE = {Signature operator on {L}ipschitz manifolds and unbounded
              {K}asparov bimodules},
 BOOKTITLE = {Operator algebras and their connections with topology and
              ergodic theory ({B}u\c steni, 1983)},
    SERIES = {Lecture Notes in Math.},
    VOLUME = {1132},
     PAGES = {254--288},
 PUBLISHER = {Springer, Berlin},
      YEAR = {1985},
      ISBN = {3-540-15643-7},
   MRCLASS = {58G10 (46L80 57R20)},
  MRNUMBER = {799572},
MRREVIEWER = {J\"urgen\ Eichhorn},
       DOI = {10.1007/BFb0074888},
       URL = {https://doi.org/10.1007/BFb0074888},
}

@article {MR0720931,
    AUTHOR = {Teleman, Nicolae},
     TITLE = {The index of signature operators on {L}ipschitz manifolds},
   JOURNAL = {Inst. Hautes \'Etudes Sci. Publ. Math.},
  FJOURNAL = {Institut des Hautes \'Etudes Scientifiques. Publications
              Math\'ematiques},
    NUMBER = {58},
      YEAR = {1983},
     PAGES = {39--78},
      ISSN = {0073-8301,1618-1913},
   MRCLASS = {58G10 (58G12)},
  MRNUMBER = {720931},
MRREVIEWER = {J\'ozef\ Dodziuk},
       URL = {http://www.numdam.org/item?id=PMIHES_1983__58__39_0},
}

@misc{kennedy2024hotspotsconjecturehigher,
	title = {On the hot spots conjecture in higher dimensions},
	url = {http://arxiv.org/abs/2410.00816},
	doi = {10.48550/arXiv.2410.00816},
	urldate = {2024-12-03},
	publisher = {arXiv},
	author = {Kennedy, James B. and Rohleder, Jonathan},
	month = oct,
	year = {2024},
	note = {arXiv:2410.00816},
}

@book {ShubinInvitation,
    AUTHOR = {Shubin, Mikhail},
     TITLE = {Invitation to partial differential equations},
    SERIES = {Graduate Studies in Mathematics},
    VOLUME = {205},
      NOTE = {Edited and with a foreword by Maxim Braverman, Robert McOwen
              and Peter Topalov,
              Translated from the 2001 Russian original},
 PUBLISHER = {American Mathematical Society, Providence, RI},
      YEAR = {2020},
     PAGES = {xvii+319},
      ISBN = {978-0-8218-3640-8},
   MRCLASS = {35-01},
  MRNUMBER = {4171481},
       DOI = {10.1090/gsm/205},
       URL = {https://doi.org/10.1090/gsm/205},
}

@misc{lotoreichik2024inequalitiesdirichletneumanneigenvalues,
	title = {Inequalities between {Dirichlet} and {Neumann} eigenvalues of the magnetic {Laplacian}},
	url = {http://arxiv.org/abs/2405.12077},
	doi = {10.48550/arXiv.2405.12077},
	urldate = {2024-12-03},
	publisher = {arXiv},
	author = {Lotoreichik, Vladimir},
	month = may,
	year = {2024},
	note = {arXiv:2405.12077},
}

@misc{freitas2024growinggapdirichletneumann,
	title = {On the (growing) gap between {Dirichlet} and {Neumann} eigenvalues},
	url = {http://arxiv.org/abs/2405.18079},
	doi = {10.48550/arXiv.2405.18079},
	urldate = {2025-01-17},
	publisher = {arXiv},
	author = {Freitas, Pedro},
	month = may,
	year = {2024},
	note = {arXiv:2405.18079},
}

@article {SafarovFilonov,
    AUTHOR = {Safarov, Yu. G. and Filonov, N. D.},
     TITLE = {Asymptotic estimates for the difference between {D}irichlet
              and {N}eumann counting functions},
   JOURNAL = {Funktsional. Anal. i Prilozhen.},
  FJOURNAL = {Funktsional\cprime ny\u i\ Analiz i ego Prilozheniya},
    VOLUME = {44},
      YEAR = {2010},
    NUMBER = {4},
     PAGES = {54--64},
      ISSN = {0374-1990,2305-2899},
   MRCLASS = {35P20 (35J05 47F05)},
  MRNUMBER = {2768564},
MRREVIEWER = {Sergey\ G.\ Pyatkov},
       DOI = {10.1007/s10688-010-0039-5},
       URL = {https://doi.org/10.1007/s10688-010-0039-5},
}

@misc{mikhail2025hodgelaplacianeigenvaluessurfacesboundary,
      title={{Hodge}-{Laplacian} Eigenvalues on Surfaces with Boundary}, 
      author={Muravyev Mikhail},
      year={2025},
      eprint={2412.19349},
      archivePrefix={arXiv},
      primaryClass={math.DG},
      url={https://arxiv.org/abs/2412.19349},
}
\bibliographystyle{alpha}

\end{document}